\documentclass{amsart}
\usepackage{setspace}
\usepackage{a4}
\usepackage{amsthm}
\usepackage{latexsym}
\usepackage{amsfonts}
\usepackage{graphicx}
\usepackage{textcomp}
\usepackage{cite}
\usepackage{enumerate}
\usepackage{amssymb}
\usepackage{hyperref}
\usepackage{amsmath}
\usepackage{tikz}
\usepackage[mathscr]{euscript}
\usepackage{mathtools}
\newtheorem{theorem}{Theorem}[section]

\newtheorem{conjecture}[theorem]{Conjecture}
\newtheorem{corollary}[theorem] {Corollary}
\newtheorem{definition}[theorem]{Definition}

\newtheorem{problem}[theorem]{Problem}

\newtheorem{question}[theorem]{Question}
\title{This is the title}
\usepackage{fancyhdr}

\begin{document}
	\vspace{0.9cm}	
\hrule\hrule\hrule\hrule\hrule
\vspace{0.3cm}	
\begin{center}
{\bf{Quaternionic Buzano-Dragomir Inequality and Deutsch Entropic Uncertainty Principle}}\\
\vspace{0.3cm}
\hrule\hrule\hrule
\vspace{0.3cm}
\textbf{K. Mahesh Krishna}\\
School of Mathematics and Natural Sciences\\
Chanakya University Global Campus\\
NH-648, Haraluru Village\\
Devanahalli Taluk, 	Bengaluru  North District\\
Karnataka State 562 110 India \\
Email: kmaheshak@gmail.com\\

Date: \today
\hrule\hrule
\end{center}

\vspace{0.5cm}
\textbf{Abstract}: We show that the  Buzano-Dragomir inequality holds for quaternionic Hilbert spaces. Using this generalization, we derive the Dragomir inequality for the quaternionic numerical radius and Deutsch entropic uncertainty principle for quaternionic Hilbert spaces.



\hrule

\hrule
\section{Introduction}
More than 50 years ago, Buzano derived the following generalization of Cauchy-Schwarz inequality \cite{BUZANO}.
\begin{theorem}  \cite{FUJII, BUZANO, STEELE} \label{BT}   (\textbf{Buzano Inequality})
	Let $\mathcal{H}$ be a complex Hilbert space. Then 
	\begin{align}\label{BI}
		|\langle \tau, h\rangle \langle h, \omega\rangle |\leq \frac{\|h\|^2(|\langle \tau, \omega \rangle|+\|\tau\|\|\omega\|)}{2}\leq \|h\|^2\|\tau\|\|\omega\|, \quad \forall h, \tau, \omega \in \mathcal{H}.
	\end{align}
In other words, 
\begin{align*}
	\left|\left\langle \tau, \frac{h}{\|h\|}\right\rangle \left \langle \frac{h}{\|h\|}, \omega \right\rangle\right|\leq \frac{|\langle \tau, \omega \rangle|+\|\tau\|\|\omega\|}{2}\leq \|\tau\|\|\omega\|, \quad \forall  \tau, \omega \in \mathcal{H}, \forall h \in \mathcal{H}\setminus \{0\}.
\end{align*}
Further, we have the following.
\begin{enumerate}[\upshape(i)]
	\item Let $\tau, \omega \in \mathcal{H}\setminus\{0\}$ be such that $\langle \tau, \omega \rangle =0$. Choose $\alpha, \beta \in \mathbb{C}$ with $|\beta|=1$. Define 
	\begin{align*}
		h_{\alpha, \beta}\coloneqq \alpha \left(\frac{\tau}{\|\tau\|}+\beta\frac{\omega}{\|\omega\|}\right).
	\end{align*}
Then 
\begin{align*}
		|\langle \tau, 	h_{\alpha, \beta}\rangle \langle 	h_{\alpha, \beta}, \omega\rangle |= \frac{\|	h_{\alpha, \beta}\|^2(|\langle \tau, \omega \rangle|+\|\tau\|\|\omega\|)}{2}.
\end{align*}
	\item  Let $\tau, \omega \in \mathcal{H}\setminus\{0\}$ be such that $\langle \tau, \omega \rangle \neq 0$. For $\alpha \in \mathbb{C}$, define 
	\begin{align*}
		h_{\alpha}\coloneqq \alpha \left(\frac{\tau}{\|\tau\|}+\frac{\langle \tau, \omega \rangle}{|\langle \tau, \omega \rangle |}\frac{\omega}{\|\omega\|}\right).
	\end{align*}
	Then 
	\begin{align*}
		|\langle \tau, 	h_{\alpha}\rangle \langle 	h_{\alpha}, \omega\rangle |= \frac{\|	h_{\alpha}\|^2(|\langle \tau, \omega \rangle|+\|\tau\|\|\omega\|)}{2}.
	\end{align*}
\end{enumerate}
\end{theorem} 
Ten years ago, in 2016, Dragomir made the following far reaching generalization of Theorem \ref{BT} \cite{DRAGOMIR2}.
\begin{theorem}   \cite{DRAGOMIR2} \label{DT} (\textbf{Buzano-Dragomir Inequality})
	Let $\mathcal{H}$ be a complex Hilbert space and  $P:\mathcal{H}\to \mathcal{H}$ be an orthogonal projection. Then 
	\begin{align}\label{DI}
		|\langle P\tau, \omega\rangle |\leq \frac{|\langle \tau, \omega \rangle|+\|\tau\|\|\omega\|}{2}, \quad \forall \tau, \omega \in \mathcal{H}.	
	\end{align}
\end{theorem}
Now note that for any nonzero $h\in \mathcal{H}$, the map 
\begin{align*}
	\mathcal{H} \ni \tau \mapsto \left\langle \tau, \frac{h}{\|h\|} \right \rangle \frac{h}{\|h\|} \in 	\mathcal{H}
\end{align*}
is a rank one orthogonal projection.  Thus Inequality (\ref{DI}) recovers Inequality (\ref{BI}). In 2017, Dragomir further improved Theorem \ref{DT} for bounded linear operators \cite{DRAGOMIR3}. 
\begin{theorem} \cite{DRAGOMIR3} \label{DIG} (\textbf{Buzano-Dragomir Inequality})
	Let $\mathcal{H}$ be a complex Hilbert space and  $T:\mathcal{H}\to \mathcal{H}$ be a bounded linear operator. Then 
\begin{align}\label{DI3}
	|\langle T\tau, T\omega\rangle |\leq \frac{\|T\|^2(|\langle \tau, \omega \rangle|+\|\tau\|\|\omega\|)}{2}, \quad \forall \tau, \omega \in \mathcal{H}.	
\end{align}	
\end{theorem}
Inequality (\ref{BI}) has  various applications, we describe few.  Recall that \cite{GUSTAFSONRAO, HALMOS, HORNJOHNSON, GOLDBERGTADMOR, GAUWUBOOK} the \textbf{numerical range} of a bounded linear operator $T:\mathcal{H} \to \mathcal{H} $ is defined as 
\begin{align*}
	W_\mathbb{C}(T)\coloneqq \{\langle Th, h\rangle: h \in \mathcal{H}, \|h\|=1\}
\end{align*}
and the \textbf{numerical radius} $T$ is defined as 
\begin{align*}
	w_\mathbb{C}(T)\coloneqq \sup_{h \in \mathcal{H}, \|h\|=1}|\langle Th, h\rangle|.
\end{align*}
For  $ \tau, \omega \in \mathcal{H}$, define the rank one operator  
\begin{align*}
	\tau \otimes \omega : \mathcal{H} \ni h \mapsto (\tau \otimes \omega)h \coloneqq \langle h, \omega \rangle \tau \in \mathcal{H}.
\end{align*}
In 1993, Fujii and Kubo derived the following result, using Inequality (\ref{BI}) \cite{FUJII}.
\begin{theorem} \cite{FUJII}
	Let $\mathcal{H}$ be a complex Hilbert space and $ \tau, \omega \in \mathcal{H}$. Then 
	\begin{align*}
	\text{(\textbf{Fujii-Kubo Equality})} \quad	w_\mathbb{C}(\tau \otimes \omega)=\frac{|\langle \tau, \omega \rangle|+\|\tau\|\|\omega\|}{2}.
	\end{align*}
\end{theorem}
Let $p(z)\coloneqq a_0+a_1z+\cdots +a_{n-1}z^{n-1}+z^n \in \mathbb{C}[z]$. A direct observation reveals that the zeros of $p$ are the eigenvalues of the \textbf{Frobenius companion matrix} 
\begin{align*}
	C_p\coloneqq 
	\begin{pmatrix}
		-a_{n-1}&-a_{n-2}&-a_{n-3}&\cdots &-a_2 &-a_1&-a_0\\
		1&0&0&\cdots &0 &0&0\\
		0&1&0&\cdots &0 &0&0\\
		\vdots &\vdots &\vdots & & \vdots &\vdots &\vdots\\
		0&0&0&\cdots &0 &0&0\\
		0&0&0&\cdots &1 &0&0\\
	 	0&0&0&\cdots &0 &1&0\\
	\end{pmatrix} \in \mathbb{M}_n(\mathbb{C})
\end{align*}
and the eigenvalues of $C_p$ are the same as the zeros of $p$. Let $R$ be the right shift matrix defined by 
\begin{align*}
	R\coloneqq 
	\begin{pmatrix}
	     0&0&0&\cdots &0 &0&0\\
		1&0&0&\cdots &0 &0&0\\
		0&1&0&\cdots &0 &0&0\\
		\vdots &\vdots &\vdots & & \vdots &\vdots &\vdots\\
		0&0&0&\cdots &0 &0&0\\
		0&0&0&\cdots &1 &0&0\\
		0&0&0&\cdots &0 &1&0\\
	\end{pmatrix} \in \mathbb{M}_n(\mathbb{C}).
\end{align*}
Define 
\begin{align*}
	{\bf{a}}\coloneqq \begin{pmatrix}
		\overline{a_{n-1}}\\
		\overline{a_{n-2}}\\
			\overline{a_{n-3}}\\
		\vdots \\
			\overline{a_{2}}\\
		\overline{a_{1}}\\
		\overline{a_{0}}\\
	\end{pmatrix}, \quad 
	{\bf{e_1}} \coloneqq \begin{pmatrix}
1\\
0\\
0\\
	\vdots \\
	0\\
	0\\
0\\
\end{pmatrix}. 
\end{align*}
Now the operator $e_1\otimes {\bf{a}}$ has matrix representation 
\begin{align*}
e_1\otimes {\bf{a}}=	\begin{pmatrix}
	a_{n-1}&a_{n-2}&a_{n-3}&\cdots &a_2 &a_1&a_0\\
	0&0&0&\cdots &0 &0&0\\
	0&0&0&\cdots &0 &0&0\\
	\vdots &\vdots &\vdots & & \vdots &\vdots &\vdots\\
	0&0&0&\cdots &0 &0&0\\
	0&0&0&\cdots &0 &0&0\\
	0&0&0&\cdots &0 &0&0\\
\end{pmatrix} \in \mathbb{M}_n(\mathbb{C}).
\end{align*}
Then 
\begin{align}\label{FK}
	C_p=R-e_1\otimes {\bf{a}}.
\end{align}
In 1993, Fujii and Kubo derived following bound for the zeros of $p$ using Inequality (\ref{BI}) and the representation (\ref{FK}) \cite{FUJII}.
\begin{theorem} \cite{FUJII}
	Let $p(z)\coloneqq a_0+a_1z+\cdots +a_{n-1}z^{n-1}+z^n \in \mathbb{C}[z]$.  If $\lambda$ is a  zero of $p$, then 
	\begin{align}\label{FNR}
		|\lambda|&\leq w_\mathbb{C}(R)+\frac{\sqrt{\sum_{j=0}^{n-1}|a_j|^2}+|a_{n-1}|}{2}=\cos\left(\frac{\pi}{n+1}\right)+\frac{\sqrt{\sum_{j=0}^{n-1}|a_j|^2}+|a_{n-1}|}{2}.
	\end{align}
\end{theorem}
For any bounded linear operator $T:\mathcal{H} \to \mathcal{H} $, the Cauchy-Schwarz inequality gives
\begin{align}\label{NR}
		w_\mathbb{C}(T)\leq \|T\|.
\end{align} 
In 2007,  Dragomir strengthened Inequality (\ref{NR})  using Buzano inequality \cite{DRAGOMIR}.
\begin{theorem} \cite{DRAGOMIR} \label{DIT} (\textbf{Dragomir Numerical Radius Inequality}) Let $T:\mathcal{H} \to \mathcal{H} $ be a bounded linear operator. Then
	\begin{align*}
		w_\mathbb{C}(T) \leq \frac{1}{\sqrt{2}}\sqrt{w_\mathbb{C}(T^2)+\|T\|^2}\leq \|T\|.
	\end{align*}
\end{theorem}
In 2016, Dragomir derived following result using Inequality (\ref{DI}).
\begin{theorem}\cite{DRAGOMIR2} \label{DI2}
Let $P:\mathcal{H} \to \mathcal{H} $ be an orthogonal projection. For any two bounded linear operators  $T, S:\mathcal{H} \to \mathcal{H} $, we have
\begin{align*}
	w_\mathbb{C}(TPS) \leq	\frac{w_\mathbb{C}(TS)+\|T\|\|S\|}{2}
\end{align*}
and 
\begin{align*}
	\|TPS\|\leq \frac{\|TS\|+\|T\|\|S\|}{2}.
\end{align*}
In particular, 
\begin{align*}
	w_\mathbb{C}(TPT) \leq	\frac{w_\mathbb{C}(T^2)+\|T\|^2}{2}
\end{align*}
and 
\begin{align*}
	\|TPT\|\leq \frac{\|T^2\|+\|T\|^2}{2}.
\end{align*}
\end{theorem}
It is observed very recently that using Inequality (\ref{BI}) we can give a very easy proof of Deutsch entropic uncertainty principle \cite{KRISHNA}.
Let $\mathcal{H}$ be a finite dimensional complex Hilbert space. Given an orthonormal basis  $\{\tau_j\}_{j=1}^n$ for $\mathcal{H}$, the \textbf{ (finite) Shannon entropy}  at a point $h \in \mathcal{H}_\tau$ is defined as 
\begin{align*}
	S_\tau (h)\coloneqq - \sum_{j=1}^{n} \left|\left \langle \frac{h}{\|h\|}, \tau_j\right\rangle \right|^2\log \left|\left \langle \frac{h}{\|h\|}, \tau_j\right\rangle \right|^2\geq 0,
\end{align*}
where $\mathcal{H}_\tau\coloneqq \{h \in \mathcal{H}: \langle h , \tau_j \rangle \neq 0, \forall 1\leq j \leq n\}$ \cite{DEUTSCH}. 
\begin{theorem}\cite{DEUTSCH} (\textbf{Deutsch Entropic Uncertainty Principle})  \label{DU}
	Let $\{\tau_j\}_{j=1}^n$,  $\{\omega_j\}_{j=1}^n$ be two orthonormal bases for a  finite dimensional Hilbert space $\mathcal{H}$. Then 
	\begin{align*}
		2 \log n \geq S_\tau (h)+S_\omega (h)\geq -2 \log \left(\frac{1+\displaystyle \max_{1\leq j, k \leq n}|\langle\tau_j , \omega_k\rangle|}{2}\right)	\geq 0, \quad \mbox{for all } h \in \mathcal{H}_\tau \cap \mathcal{H}_\omega.
	\end{align*}
\end{theorem}
In this  article, we show that Buzano and Buzano-Dragomir inequalities hold for quaternionic Hilbert spaces. Using quaternionic Buzano inequality we derive quaternionic Deutsch entropic uncertainty principle. We end the introduction by recalling the definition of  quaternionic Hilbert spaces. Let  
\begin{align*}
	\mathbb{H}\coloneqq \{a+ib+jc+kd: a, b, c, d \in \mathbb{R}, i^2=j^2=k^2=ijk=-1\}
\end{align*}
be the skew-field of quaternion numbers with the conjugation
\begin{align*}
	\overline{a+ib+jc+kd}\coloneqq a-ib-jc-kd, \quad \forall a+ib+jc+kd \in 	\mathbb{H}
\end{align*}
and the modulus 
\begin{align*}
	|a+ib+jc+kd|\coloneqq \sqrt{a^2+b^2+c^2+d^2}, \quad \forall a+ib+jc+kd \in 	\mathbb{H}.	
\end{align*}
\begin{definition}\cite{GHILONIMORETTIPEROTTI}
	A right quaternionic  vector space 	 $\mathcal{Q}$  over $\mathbb{H}$ is said to be a (right)  \textbf{quaternionic Hilbert space} if there exists a  map (called as quaternionic inner product) $ \langle \cdot, \cdot \rangle: \mathcal{Q}\times \mathcal{Q} \to \mathbb{H}$ such that the following hold. 
	\begin{enumerate}[\upshape(i)]
		\item $\langle x, x \rangle  \geq 0$, $\forall x \in \mathcal{Q}$. If $x \in  \mathcal{Q}$ satisfies $\langle x, x \rangle  =0 $, then $x=0$.
		\item $\langle x+y, z \rangle  =\langle x, z \rangle+\langle y, z \rangle$, $\forall x,y,z \in \mathcal{Q}$.
		\item  $\langle x, yq \rangle  =\langle x, y \rangle q$, $\forall x,y \in \mathcal{Q}$, $\forall q \in \mathbb{H}$.
		\item $\langle x, y \rangle=\overline{\langle y,x \rangle}$, $\forall x,y \in \mathcal{Q}$.
		\item $\mathcal{Q}$ is complete w.r.t. the norm $\|x\|\coloneqq \sqrt{\langle x, x \rangle}$,  $\forall x \in \mathcal{Q}$.
	\end{enumerate}
\end{definition}

\section{Quaternionic Buzano-Dragomir Inequality and Deutsch Entropic Uncertainty Principle}
In the article,  $\mathcal{Q}$ denotes a right quaternionic Hilbert space. As in the tradition of quaternionic analysis, in the entire paper, given $x \in \mathcal{Q}$ and a  $r \in \mathbb{R}\setminus \{0\}$,  by $x/r$, we mean $x r^{-1}$. Note that this is meaningful only for division by non-zero reals and not for arbitrary quaternions as quaternions are not commutative. We first derive quaternionic version of Inequality (\ref{BI}). 
\begin{theorem}     (\textbf{Quaternionic Buzano Inequality})
	Let $\mathcal{Q}$ be a quaternionic Hilbert space. Then 
	\begin{align}\label{QBI}
	|\langle \tau, x\rangle \langle x, \omega\rangle |\leq \frac{\|x\|^2(|\langle \tau, \omega \rangle|+\|\tau\|\|\omega\|)}{2}, \quad \forall x, \tau, \omega \in \mathcal{Q}.
\end{align}
In other words, 
\begin{align*}
	\left|\left\langle \tau, \frac{x}{\|x\|}\right\rangle \left \langle \frac{x}{\|x\|}, \omega \right\rangle\right|\leq \frac{|\langle \tau, \omega \rangle|+\|\tau\|\|\omega\|}{2}\leq \|\tau\|\|\omega\|, \quad \forall  \tau, \omega \in \mathcal{Q}, \forall x \in \mathcal{Q}\setminus \{0\}.
\end{align*}
Further, we have following.
\begin{enumerate}[\upshape(i)]
	\item Let $\tau, \omega \in \mathcal{Q}\setminus\{0\}$ be such that $\langle \tau, \omega \rangle =0$. Choose $\alpha, \beta \in \mathbb{H}$ with $|\beta|=1$. Define 
	\begin{align*}
		x_{\alpha, \beta}\coloneqq \left(\frac{\tau}{\|\tau\|}+\frac{\omega}{\|\omega\|}\beta\right)\alpha.
	\end{align*}
	Then 
	\begin{align*}
		|\langle \tau, 	x_{\alpha, \beta}\rangle \langle 	x_{\alpha, \beta}, \omega\rangle |= \frac{\|	x_{\alpha, \beta}\|^2(|\langle \tau, \omega \rangle|+\|\tau\|\|\omega\|)}{2}.
	\end{align*}
	\item  Let $\tau, \omega \in \mathcal{Q}\setminus\{0\}$ be such that $\langle \tau, \omega \rangle \neq 0$. For $\alpha \in \mathbb{H}$, define 
	\begin{align*}
		x_{\alpha}\coloneqq  \left(\frac{\tau}{\|\tau\|}+\frac{\omega}{\|\omega\|}\frac{\langle \omega, \tau \rangle}{|\langle \omega, \tau \rangle |}\right)\alpha.
	\end{align*}
	Then 
	\begin{align*}
		|\langle \tau, 	x_{\alpha}\rangle \langle 	x_{\alpha}, \omega\rangle |= \frac{\|	x_{\alpha}\|^2(|\langle \tau, \omega \rangle|+\|\tau\|\|\omega\|)}{2}.
	\end{align*}
\end{enumerate}
\end{theorem}
\begin{proof}
Let $x, \tau, \omega \in \mathcal{Q}$. We first see 
	\begin{align*}
		|2\langle \tau, x \rangle \langle x, \omega \rangle -\langle \tau, \omega\rangle \|x\|^2| &=	|\langle \tau, 2x \langle x, \omega \rangle \rangle -\langle \tau, \omega\|x\|^2\rangle |\\
		&=|\langle \tau, 2x \langle x, \omega\rangle - \omega\|x\|^2\rangle |\\
		&\leq \|\tau\|\|2x \langle x, \omega\rangle - \omega\|x\|^2\|.
	\end{align*}
We now find 
\begin{align*}
	\|2x \langle x, \omega\rangle - \omega\|x\|^2\|^2&=\langle 2x \langle x, \omega\rangle - \omega\|x\|^2, 2x \langle x, \omega\rangle - \omega\|x\|^2\rangle \\
	&=\langle 2x \langle x, \omega\rangle , 2x \langle x, \omega\rangle \rangle-\langle 2x \langle x, \omega\rangle ,  \omega\|x\|^2\rangle-\langle \omega\|x\|^2, 2x \langle x, \omega\rangle\rangle+\langle \omega\|x\|^2, \omega\|x\|^2\rangle\\
	&=4\|x\|^2|\langle x, \omega\rangle|^2-2\|x\|^2|\langle x, \omega\rangle|^2-2\|x\|^2|\langle x, \omega\rangle|^2+\|\omega\|^2\|x\|^4\\
	&=\|\omega\|^2\|x\|^4.
\end{align*}
Hence 
\begin{align*}
	\|2x \langle x, \omega\rangle - \omega\|x\|^2\|=\|\omega\|\|x\|^2.
\end{align*}
Therefore 
\begin{align*}
	2|\langle \tau, x \rangle \langle x, \omega \rangle| -|\langle \tau, \omega\rangle \|x\|^2|&\leq 	|2\langle \tau, x \rangle \langle x, \omega \rangle -\langle \tau, \omega\rangle \|x\|^2|\\
	&\leq \|\tau\|\|\omega\|\|x\|^2.
\end{align*}
Other assertions in the statement come from direct computation. 	
\end{proof}
We now derive quaternionic version of Theorem \ref{DT}. Our proof is motivated from the proof by Dragomir \cite{DRAGOMIR2}.
\begin{theorem}     (\textbf{Quaternionic Buzano-Dragomir Inequality})
	Let $\mathcal{Q}$ be a quaternionic Hilbert space and  $P:\mathcal{Q}\to \mathcal{Q}$ be an orthogonal projection. Then 
	\begin{align}\label{BDI}
		|\langle P\tau, \omega\rangle |\leq \frac{|\langle \tau, \omega \rangle|+\|\tau\|\|\omega\|}{2}, \quad \forall \tau, \omega \in \mathcal{Q}.	
	\end{align}
\end{theorem}
\begin{proof}
	Let $I_\mathcal{Q}$ be the identity operator on $\mathcal{Q}$.	Since $P$ is an orthogonal projection, 
	\begin{align*}
		\|2P\tau-\tau\|&=\sqrt{\langle 2P\tau-\tau, 2P\tau-\tau \rangle}\\
		&=\sqrt{4\langle P\tau, \tau \rangle-2\langle P\tau, \tau \rangle-2\langle \tau, P\tau \rangle+\langle \tau, \tau \rangle}=\|\tau\|.
	\end{align*}
	Therefore 
	\begin{align*}
		|\langle 2P\tau, \omega \rangle|-|\langle \tau, \omega \rangle|&\leq |\langle 2P\tau, \omega \rangle-\langle \tau, \omega \rangle|=|\langle 2P\tau- \tau, \omega \rangle|\\
		&\leq 	\|2P\tau-\tau\|\|\omega\|=\|\tau\|\|\omega\|.
	\end{align*}	
\end{proof}
Note that we can give a direct proof of Inequality (\ref{BDI}) using Inequality (\ref{QBI}). For Hilbert spaces, this has been observed in 2022  by Sababheh, Moradi  and Heydarbeygi \cite{SMH}. Let  $\mathcal{Q}$ be a quaternionic Hilbert space and  $P:\mathcal{Q}\to \mathcal{Q}$ be an orthogonal projection. Let $\tau, \omega \in \mathcal{Q}$. Then using Inequality (\ref{QBI}) we get 
\begin{align*}
	|\langle P\tau, \omega \rangle |&=\frac{1}{\|P\tau\|^2}\|P\tau\|^2|\langle P\tau, \omega \rangle |=\frac{1}{\|P\tau\|^2}\langle P\tau, P\tau \rangle \langle P\tau, \omega \rangle |\\
	&=\frac{1}{\|P\tau\|^2}|\langle \tau, P\tau \rangle \langle P\tau, \omega \rangle |\leq \frac{1}{\|P\tau\|^2}\frac{\|P\tau\|^2(|\langle \tau, \omega \rangle|+\|\tau\|\|\omega\|)}{2}\\
	&=\frac{|\langle \tau, \omega \rangle|+\|\tau\|\|\omega\|}{2}, 
\end{align*}
which is Inequality (\ref{BDI}). We now derive Theorem \ref{DIG} for quaternionic Hilbert spaces. Our proof is motivated from the proof by Dragomir \cite{DRAGOMIR3}.
\begin{theorem}     (\textbf{Quaternionic Buzano-Dragomir Inequality})
	Let $\mathcal{Q}$ be a quaternionic Hilbert space and  $T:\mathcal{Q}\to \mathcal{Q}$ be a bounded linear operator. Then 
	\begin{align*}
		|\langle T\tau, T\omega\rangle |\leq \frac{\|T\|^2(|\langle \tau, \omega \rangle|+\|\tau\|\|\omega\|)}{2}, \quad \forall \tau, \omega \in \mathcal{Q}.	
	\end{align*}
\end{theorem}
\begin{proof}
	Let $T:\mathcal{Q}\to \mathcal{Q}$ be a bounded linear operator and $\tau, \omega \in \mathcal{Q}$. Define 
	\begin{align*}
		[\cdot, \cdot]:  \mathcal{Q} \times  \mathcal{Q} \ni (x,y)\mapsto [x,y]\coloneqq \langle x, y \rangle \|T\|^2-\langle Tx, Ty \rangle \in \mathbb{H}.
	\end{align*}
Then $	[\cdot, \cdot]$ is a semi-inner  product. By applying Cauchy-Schwarz inequality to this semi-inner product, we get 
\begin{align*}
	|[x,y]|^2\leq [x, x][y,y], \quad \forall x, y \in \mathcal{Q}.
\end{align*}
In particular, 
\begin{align*}
		&\left|[\tau, \omega]\right|^2\leq [\tau, \tau][\omega, \omega].
	\end{align*}
Hence 
\begin{align*}
		 \big|\langle \tau, \omega \rangle \|T\|^2-\langle T\tau, T\omega \rangle\big|^2 \leq (\|\tau\|^2\|T\|^2-\|T\tau\|^2)(\|\omega\|^2\|T\|^2-\|T\omega\|^2).
\end{align*}
Note that we have 
\begin{align*}
	(b^2-a^2)(d^2-c^2)\leq (bd-ac)^2, \quad \forall a, b, c, d \in [0,\infty).
\end{align*}
Hence 
\begin{align*}
	\big|\langle \tau, \omega \rangle \|T\|^2-\langle T\tau, T\omega \rangle\big|^2 &\leq (\|\tau\|^2\|T\|^2-\|T\tau\|^2)(\|\omega\|^2\|T\|^2-\|T\omega\|^2)\\
	&\leq (\|\tau\|\|T\|\|\omega\|\|T\|-\|T\tau\|\|T\omega\|)^2.
\end{align*}
Taking square root in the previous equation, 
\begin{align*}
	\big|\langle \tau, \omega \rangle \|T\|^2-\langle T\tau, T\omega \rangle\big| &\leq|\|\tau\|\|T\|\|\omega\|\|T\|-\|T\tau\|\|T\omega\||\\
	&=\|\tau\|\|T\|^2\|\omega\|-\|T\tau\|\|T\omega\|.
\end{align*}
Using triangle inequality and Cauchy-Schwarz inequality, we get 
\begin{align*}
		|\langle T\tau, T\omega \rangle|-|\langle \tau, \omega \rangle| \|T\|^2\leq \|\tau\|\|T\|^2\|\omega\|-\|T\tau\|\|T\omega\|\leq \|\tau\|\|T\|^2\|\omega\|-|\langle T\tau, T\omega \rangle|.
\end{align*}
Therefore 
\begin{align*}
	2 	|\langle T\tau, T\omega \rangle|\leq \|T\|^2\|\tau\|\|\omega\|+|\langle \tau, \omega \rangle|\|T\|^2.
\end{align*}
\end{proof}
Similar to Hilbert space case, the \textbf{quaternionic numerical range} \cite{YEUNG, KIPPENHAHN, SO, SOTHOMSON, ZHANG, SOTHOMPSONZHANG, THOMPSON, CARVALHODIOGOMENDES, CARVALHODIOGOMENDES2, CARVALHODIOGOMENDES3, SO2, YEUNGSIU}
of a quaternionic bounded linear operator $T:\mathcal{Q} \to \mathcal{Q} $ is defined as 
\begin{align*}
	W_\mathbb{H}(T)\coloneqq \{\langle x,Tx\rangle: x \in \mathcal{Q}, \|x\|=1\}
\end{align*}
and the \textbf{quaternionic numerical radius} of  $T$ is defined as 
\begin{align*}
	w_\mathbb{H}(T)\coloneqq \sup_{x \in \mathcal{Q}, \|x\|=1}|\langle Tx, x\rangle|.
\end{align*}
For  $ \tau, \omega \in \mathcal{Q}$, define the rank one operator  
\begin{align*}
	\tau \otimes \omega : \mathcal{Q} \ni x \mapsto (\tau \otimes \omega)x \coloneqq \tau \langle \omega, x \rangle \in \mathcal{Q}.
\end{align*}
Using Inequality (\ref{QBI}), we  have the following theorem. 
\begin{theorem} 
	Let $\mathcal{Q}$ be a quaternionic  Hilbert space and $ \tau, \omega \in \mathcal{Q}$. Then 
	\begin{align}\label{RE}
		w_\mathbb{H}(\tau \otimes \omega)=\frac{|\langle \tau, \omega \rangle|+\|\tau\|\|\omega\|}{2}.
	\end{align}
\end{theorem}
\begin{proof}
Let $ \tau, \omega \in \mathcal{Q}$. Let $x \in \mathcal{Q}$	with $\|x\|=1$. Then 
\begin{align*}
	|\langle (\tau \otimes \omega)x, x \rangle| =|\langle \tau \langle \omega, x \rangle, x \rangle |=|\langle x, \omega \rangle \langle \tau, x \rangle |\leq \frac{|\langle \tau, \omega \rangle|+\|\tau\|\|\omega\|}{2}.
\end{align*}
Therefore 
\begin{align*}
		w_\mathbb{H}(\tau \otimes \omega)\leq \frac{|\langle \tau, \omega \rangle|+\|\tau\|\|\omega\|}{2}.
\end{align*}
We next need to show the equality condition in (\ref{RE}). We can assume that $\tau \neq 0$ and $\omega \neq 0$. We have to consider two cases.\\
 Case (i): $\langle \tau, \omega \rangle =0$. Define 
\begin{align*}
	y \coloneqq \frac{1}{\sqrt{2}}\left(\frac{\tau}{\|\tau\|}+\frac{\omega}{\|\omega\|}\right). 
\end{align*}
Then $\|y\|=1$ and 
\begin{align*}
	|\langle (\tau \otimes \omega)y, y \rangle| &=|\langle \tau \langle \omega, y \rangle, y \rangle |=|\langle y, \omega \rangle|| \langle \tau, y \rangle |\\
	&=\frac{1}{2}\left|\left \langle \frac{\tau}{\|\tau\|}+\frac{\omega}{\|\omega\|}, \omega \right \rangle \right|\left|\left \langle \tau, \frac{\tau}{\|\tau\|}+\frac{\omega}{\|\omega\|} \right \rangle \right|\\
	&=\frac{1}{2}\|\omega\|\|\tau\|=\frac{0+\|\tau\|\|\omega\|}{2}=\frac{|\langle \tau, \omega \rangle|+\|\tau\|\|\omega\|}{2}.
\end{align*}
Case (ii): $\langle \tau, \omega \rangle \neq 0$. Define 
\begin{align*}
	z \coloneqq \sqrt{\frac{\|\tau\|\|\omega\|}{2(|\langle \tau, \omega \rangle|+\|\tau\|\|\omega\|)}}\left(\frac{\tau}{\|\tau\|}+\frac{\omega \langle \omega, \tau \rangle }{\|\omega\||\langle \omega, \tau \rangle|}\right). 
\end{align*}
Then 
\begin{align*}
	\|z\|^2&=\frac{\|\tau\|\|\omega\|}{2 (|\langle \tau, \omega \rangle|+\|\tau\|\|\omega\|)}\left\|\frac{\tau}{\|\tau\|}+\frac{\omega \langle \omega, \tau \rangle }{\|\omega\||\langle \omega, \tau \rangle|}\right\|^2\\
	&=\frac{\|\tau\|\|\omega\|}{2 (|\langle \tau, \omega \rangle|+\|\tau\|\|\omega\|)}\left\langle \frac{\tau}{\|\tau\|}+\frac{\omega \langle \omega, \tau \rangle }{\|\omega\||\langle \omega, \tau \rangle|}, \frac{\tau}{\|\tau\|}+\frac{\omega \langle \omega, \tau \rangle }{\|\omega\||\langle \omega, \tau \rangle|}\right \rangle \\
	&=\frac{\|\tau\|\|\omega\|}{2(|\langle \tau, \omega \rangle|+\|\tau\|\|\omega\|)}\left(2+2\frac{|\langle \tau, \omega \rangle |}{\|\tau\|\|\omega\|}\right)=1
\end{align*}
 and 
\begin{align*}
	|\langle (\tau \otimes \omega)z, z \rangle| &=|\langle \tau \langle \omega, z \rangle, z \rangle |=|\langle z, \omega \rangle|| \langle \tau, z \rangle |\\
	&=\frac{\|\tau\|\|\omega\|}{2 (|\langle \tau, \omega \rangle|+\|\tau\|\|\omega\|)}\left|\left \langle \frac{\tau}{\|\tau\|}+\frac{\omega\langle \omega, \tau \rangle }{\|\omega\||\langle \omega, \tau \rangle |}, \omega \right \rangle \right|\left|\left \langle \tau, \frac{\tau}{\|\tau\|}+\frac{\omega\langle \omega, \tau \rangle}{\|\omega\||\langle \omega, \tau \rangle |} \right \rangle \right|\\
	&=\frac{\|\tau\|\|\omega\|}{2 (|\langle \tau, \omega \rangle|+\|\tau\|\|\omega\|)}\left|\frac{\langle \tau, \omega \rangle}{\|\tau\|}+\frac{\langle \tau, \omega \rangle \|\omega \|}{|\langle \omega, \tau \rangle| }\right|\left|\|\tau\|+\frac{|\langle \omega, \tau \rangle |}{\|\omega\|}\right|\\
	&=\left(\frac{\|\tau\|\|\omega\|}{2 (|\langle \tau, \omega \rangle|+\|\tau\|\|\omega\|)}\right)|\langle \tau, \omega \rangle |\left(\frac{1}{\|\tau\|}+\frac{\|\omega\|}{|\langle \omega, \tau \rangle |}\right)\left(\frac{\|\tau\|\|\omega\|+|\langle \omega, \tau \rangle|}{\|\omega\|}\right)\\
	&=\left(\frac{\|\tau\|\|\omega\|}{2 (|\langle \tau, \omega \rangle|+\|\tau\|\|\omega\|)}\right)|\langle \tau, \omega \rangle |\left(\frac{|\langle \omega, \tau \rangle |+\|\tau\|\|\omega\|}{\|\tau\||\langle \omega, \tau\rangle|}\right)\left(\frac{\|\tau\|\|\omega\|+|\langle \omega, \tau \rangle|}{\|\omega\|}\right)\\
	&=\frac{|\langle \tau, \omega \rangle|+\|\tau\|\|\omega\|}{2}.
\end{align*}
\end{proof}
Let $p(z)\coloneqq a_0+a_1z+\cdots +a_{n-1}z^{n-1}+z^n \in \mathbb{H}[z]$. A direct observation reveals that the zeros of $p$ are the left (also right) eigenvalues of the Frobenius companion matrix 
\begin{align*}
	C_p\coloneqq 
	\begin{pmatrix}
		-a_{n-1}&-a_{n-2}&-a_{n-3}&\cdots &-a_2 &-a_1&-a_0\\
		1&0&0&\cdots &0 &0&0\\
		0&1&0&\cdots &0 &0&0\\
		\vdots &\vdots &\vdots & & \vdots &\vdots &\vdots\\
		0&0&0&\cdots &0 &0&0\\
		0&0&0&\cdots &1 &0&0\\
		0&0&0&\cdots &0 &1&0\\
	\end{pmatrix} \in \mathbb{M}_n(\mathbb{H})
\end{align*}
and the left (also right) eigenvalues of $C_p$ are the same as the zeros of $p$. Let $R$ be the right shift matrix defined by 
\begin{align*}
	R\coloneqq 
	\begin{pmatrix}
		0&0&0&\cdots &0 &0&0\\
		1&0&0&\cdots &0 &0&0\\
		0&1&0&\cdots &0 &0&0\\
		\vdots &\vdots &\vdots & & \vdots &\vdots &\vdots\\
		0&0&0&\cdots &0 &0&0\\
		0&0&0&\cdots &1 &0&0\\
		0&0&0&\cdots &0 &1&0\\
	\end{pmatrix} \in \mathbb{M}_n(\mathbb{H}).
\end{align*}
Define 
\begin{align*}
	{\bf{a}}\coloneqq \begin{pmatrix}
		\overline{a_{n-1}}\\
		\overline{a_{n-2}}\\
		\overline{a_{n-3}}\\
		\vdots \\
		\overline{a_{2}}\\
		\overline{a_{1}}\\
		\overline{a_{0}}\\
	\end{pmatrix}, \quad 
	{\bf{e_1}} \coloneqq \begin{pmatrix}
		1\\
		0\\
		0\\
		\vdots \\
		0\\
		0\\
		0\\
	\end{pmatrix}. 
\end{align*}
Then 
\begin{align}\label{QFK}
	C_p=R-e_1\otimes {\bf{a}}.
\end{align}

\begin{theorem} 
	Let $p(z)\coloneqq a_0+a_1z+\cdots +a_{n-1}z^{n-1}+z^n \in \mathbb{H}[z]$.  If $\lambda$ is a  zero of $p$, then 
	\begin{align*}
		|\lambda|\leq \frac{w_\mathbb{H}(R)+\sqrt{\sum_{j=0}^{n-1}|a_j|^2}+|a_{n-1}|}{2}.
	\end{align*}
\end{theorem}
\begin{proof}
Let $\lambda$ be a  zero of $p$. Then $\lambda$ is a left  eigenvalue of $C_p$.  Let $x \in \mathbb{H}^n$ be such that $C_px=\lambda x$ and $\|x\|=1$. Using (\ref{QFK}), we get 
	\begin{align*}
	|\lambda|\leq w_\mathbb{H}(C_p)&=w_\mathbb{H}(R-e_1\otimes 	{\bf{a}}) \\
	&\leq w_\mathbb{H}(R)+w_\mathbb{H}(e_1\otimes {\bf{a}})\\
	&=\frac{w_\mathbb{H}(R)+\sqrt{\sum_{j=0}^{n-1}|a_j|^2}+|a_{n-1}|}{2}.
	\end{align*}
\end{proof}
Equality term in (\ref{FNR}) comes from the equality of numerical radius of right shift matrix obtained by Davidson and Holbrook \cite{DAVIDSONHOLBROOK} (also see \cite{HAAGERUPHARPE, KARAEV}). We are unable to do this for the right shift matrix over quaternions. 
We now derive Theorem \ref{DIT} over quaternions.
\begin{theorem} (\textbf{Quaternionic Dragomir Numerical Radius Inequality}) Let $\mathcal{Q}$ be a quaternionic Hilbert space and  $T:\mathcal{Q} \to \mathcal{Q} $ be a bounded linear operator. Then
	\begin{align*}
		w_\mathbb{H}(T) \leq \frac{1}{\sqrt{2}}\sqrt{w_\mathbb{H}(T^2)+\|T\|^2}.
	\end{align*}
\end{theorem}
\begin{proof}
	Let $x\in  \mathcal{Q}$ with $\|x\|=1$. Then using inequality (\ref{QBI}), we get 
	\begin{align*}
	|\langle Tx, x\rangle|^2&=|\langle Tx, x\rangle \langle x, T^*x\rangle|\leq \frac{\|x\|^2(|\langle Tx, T^*x \rangle|+\|Tx\|\|T^*x\|)}{2}\\
	&= \frac{|\langle T^2x, x \rangle|+\|Tx\|\|T^*x\|}{2}\leq  \frac{|\langle T^2x, x \rangle|+\|T\|\|x\|\|T^*\|\|x\|}{2}\\
	&=\frac{|\langle T^2x, x \rangle|+\|T\|\|T^*\|}{2}=\frac{|\langle T^2x, x \rangle|+\|T\|^2}{2}\\
	&\leq \frac{\sup_{y \in \mathcal{Q}, \|y\|=1}|\langle T^2y, y \rangle|+\|T\|^2}{2}=\frac{w_\mathbb{H}(T^2)+\|T\|^2}{2}.
	\end{align*}
Therefore 
\begin{align*}
		w_\mathbb{H}(T)= \sup_{x \in \mathcal{Q}, \|x\|=1}|\langle Tx, x\rangle|\leq \frac{1}{\sqrt{2}}\sqrt{w_\mathbb{H}(T^2)+\|T\|^2}.
\end{align*}
\end{proof}

We next derive quaternionic version of Deutsch entropic uncertainty principle. We derive it for even continuous frames which are more general than orthonormal bases. 
\begin{definition} \cite{KHOKULANTHIRULOGASANTHARMURALEETHARAN, KHOKULANTHIRULOGASANTHAR}
	Let 	$(\Omega, \mu)$ be a measure space. A collection   $\{\tau_\alpha\}_{\alpha\in \Omega}$ in 	a  quaternionic Hilbert  space $\mathcal{Q}$ is said to be a \textbf{continuous Parseval frame}  for $\mathcal{Q}$ if the following conditions hold.
	\begin{enumerate}[\upshape(i)]
		\item For each $x \in \mathcal{Q}$, the map $\Omega \ni \alpha \mapsto \langle  \tau_\alpha, x \rangle \in \mathbb{H}$ is measurable.
		\item 
		\begin{align*}
			\|x\|^2=\int\limits_{\Omega}|\langle  \tau_\alpha, x \rangle|^2\,d\mu(\alpha), \quad \forall x \in \mathcal{Q}.
		\end{align*}
	\end{enumerate}
\end{definition}
We consider the following subclass of continuous Parseval frames.
\begin{definition}\label{1B}
	A continuous Parseval frame $\{\tau_\alpha\}_{\alpha\in \Omega}$   for $\mathcal{Q}$ 	is said to be \textbf{1-bounded} if 
	\begin{align*}
		\|\tau_\alpha\|\leq 1, \quad \forall \alpha \in \Omega.
	\end{align*}
\end{definition}
Given a continuous 1-bounded Parseval frame   $\{\tau_\alpha\}_{\alpha\in \Omega}$ for $\mathcal{Q}$,   we define the \textbf{continuous Shannon entropy} at a point $x \in \mathcal{Q}_\tau$ as 
\begin{align*}
	S_\tau(x)\coloneqq  -\int\limits_{\Omega}\left|\left \langle \tau_\alpha, \frac{x}{\|x\|}\right\rangle \right|^2\log \left|\left \langle \tau_\alpha, \frac{x}{\|x\|}\right\rangle \right|^2\,d\mu(\alpha)\geq 0,
\end{align*}
where $\mathcal{Q}_\tau \coloneqq \{x \in \mathcal{Q}: \langle \tau_\alpha, x \rangle \neq 0, \forall \alpha \in \Omega\}$. We now have quaternionic version of a generalization of  Theorem \ref{DU}.
\begin{theorem}(\textbf{Quaternionic  Deutsch Uncertainty Principle})\label{CDUP} Let $(\Omega, \mu)$,  $(\Delta, \nu)$ be   measure spaces  and $\{\tau_\alpha\}_{\alpha\in \Omega}$, $\{\omega_\beta\}_{\beta \in \Delta}$ be  	1-bounded continuous Parseval frames for a  quaternionic Hilbert  space $\mathcal{Q}$.
	Then   
	\begin{align}\label{CDSP}
		\log (\mu(\Omega)\nu(\Delta))\geq	S_\tau(x)+S_\omega (x)\geq -2  \log \left(\frac{1+\displaystyle \sup_{\alpha \in \Omega, \beta \in \Delta}|\langle\tau_\alpha , \omega_\beta\rangle|}{2}\right)	\geq 0, \quad \forall x \in \mathcal{Q}_\tau \cap \mathcal{Q}_\omega.
	\end{align}
\end{theorem}
\begin{proof}
	Since 	$1=\int\limits_{\Omega}\left|\left \langle \frac{x}{\|x\|}, \tau_\alpha\right\rangle \right|^2\,d\mu(\alpha)$ for all $x \in \mathcal{Q} \setminus \{0\}$, $1=\int\limits_{\Delta}\left|\left \langle \frac{x}{\|x\|}, \omega_\beta\right\rangle \right|^2\,d\nu(\beta)$ for all $x \in \mathcal{Q} \setminus \{0\}$ and  $\log$ is concave, using Jensen's inequality  we get 
	\begin{align*}
		&S_\tau(x)+S_\omega (x)=\\
		&\int \limits_{\Omega}\left|\left \langle \frac{x}{\|x\|}, \tau_\alpha\right\rangle \right|^2\log \left(\frac{1}{\left|\left \langle \frac{x}{\|x\|}, \tau_\alpha\right\rangle \right|^2}\right)\,d\mu(\alpha)+\int \limits_{\Delta}\left|\left \langle \frac{x}{\|x\|}, \omega_\beta\right\rangle \right|^2\log \left(\frac{1}{\left|\left \langle \frac{x}{\|x\|}, \omega_\beta\right\rangle \right|^2}\right)\,d\nu(\beta)\\
		&\leq \log \left(\int \limits_{\Omega}\left|\left \langle \frac{x}{\|x\|}, \tau_\alpha\right\rangle \right|^2\frac{1}{\left|\left \langle \frac{x}{\|x\|}, \tau_\alpha\right\rangle \right|^2}\,d\mu(\alpha)\right)+\log \left(\int \limits_{\Delta}\left|\left \langle \frac{x}{\|x\|}, \omega_\beta\right\rangle \right|^2\frac{1}{\left|\left \langle \frac{x}{\|x\|}, \omega_\beta\right\rangle \right|^2}\,d\nu(\beta)\right)\\
		&=\log (\mu(\Omega))+\log\nu (\Delta))=\log (\mu(\Omega)\nu (\Delta)), \quad \forall x \in \mathcal{Q}_\tau \cap \mathcal{Q}_\omega.
	\end{align*}
	Let $x \in \mathcal{Q}_\tau \cap \mathcal{Q}_\omega$. Then using Inequality (\ref{QBI})  we get 
	\begin{align*}
		S_\tau(x)+S_\omega (x)&=-\int\limits_{\Delta}\int\limits_{\Omega}\left|\left \langle \frac{x}{\|x\|}, \tau_\alpha\right\rangle \right|^2\left|\left \langle \frac{x}{\|x\|}, \omega_\beta\right\rangle \right|^2\left[\log \left|\left \langle \frac{x}{\|x\|}, \tau_\alpha\right\rangle \right|^2+\log \left|\left \langle \frac{x}{\|x\|}, \omega_\beta\right\rangle \right|^2\right]\,d\mu(\alpha)\,d\nu(\beta)\\
		&=-\int\limits_{\Delta}\int\limits_{\Omega}\left|\left \langle \frac{x}{\|x\|}, \tau_\alpha\right\rangle \right|^2\left|\left \langle \frac{x}{\|x\|}, \omega_\beta\right\rangle \right|^2\log \left|\left \langle \frac{x}{\|x\|}, \tau_\alpha\right\rangle \left \langle \frac{x}{\|x\|}, \omega_\beta\right\rangle\right|^2\,d\mu(\alpha)\,d\nu(\beta)\\
		&=-2\int\limits_{\Delta}\int\limits_{\Omega}\left|\left \langle \frac{x}{\|x\|}, \tau_\alpha\right\rangle \right|^2\left|\left \langle \frac{x}{\|x\|}, \omega_\beta\right\rangle \right|^2\log \left|\left \langle \frac{x}{\|x\|}, \tau_\alpha\right\rangle \left \langle \frac{x}{\|x\|}, \omega_\beta\right\rangle\right|\,d\mu(\alpha)\,d\nu(\beta)\\
		&\geq -2\int\limits_{\Delta}\int\limits_{\Omega}\left|\left \langle \frac{x}{\|x\|}, \tau_\alpha\right\rangle \right|^2\left|\left \langle \frac{x}{\|x\|}, \omega_\beta\right\rangle \right|^2\log \left(\left\|\frac{x}{\|x\|}\right\|^2\frac{\|\tau_\alpha\|\|\omega_\beta\|+|\langle\tau_\alpha, \omega_\beta \rangle|}{2}\right)	\,d\mu(\alpha)\,d\nu(\beta)\\
		& =-2\int\limits_{\Delta}\int\limits_{\Omega}\left|\left \langle \frac{x}{\|x\|}, \tau_\alpha\right\rangle \right|^2\left|\left \langle \frac{x}{\|x\|}, \omega_\beta\right\rangle \right|^2\log \left(\frac{1+|\langle\tau_\alpha, \omega_\beta \rangle|}{2}\right)	\,d\mu(\alpha)\,d\nu(\beta)\\
		&\geq -2\int\limits_{\Delta}\int\limits_{\Omega}\left|\left \langle \frac{x}{\|x\|}, \tau_\alpha\right\rangle \right|^2\left|\left \langle \frac{x}{\|x\|}, \omega_\beta\right\rangle \right|^2\log \left(\frac{1+\displaystyle \sup_{\alpha \in \Omega, \beta \in \Delta}|\langle\tau_\alpha, \omega_\beta \rangle|}{2}\right)	\,d\mu(\alpha)\,d\nu(\beta)\\
		&=-2\log \left(\frac{1+\displaystyle \sup_{\alpha \in \Omega, \beta \in \Delta}|\langle\tau_\alpha, \omega_\beta \rangle|}{2}\right)\int\limits_{\Delta}\int\limits_{\Omega}\left|\left \langle \frac{x}{\|x\|}, \tau_\alpha\right\rangle \right|^2\left|\left \langle \frac{x}{\|x\|}, \omega_\beta\right\rangle \right|^2\,d\mu(\alpha)\,d\nu(\beta)\\
		&\geq-2\log \left(\frac{1+\displaystyle \sup_{\alpha \in \Omega, \beta \in \Delta}|\langle\tau_\alpha, \omega_\beta \rangle|}{2}\right).
	\end{align*}	
\end{proof}
\begin{corollary}
	Let $\{\tau_j\}_{j=1}^n$,  $\{\omega_k\}_{k=1}^m$ be two Parseval frames for a  finite dimensional quaternionic Hilbert space $\mathcal{Q}$. Then 
	\begin{align*}
		 \log n+\log m \geq S_\tau (x)+S_\omega (x)\geq -2 \log \left(\frac{1+\displaystyle \max_{1\leq j \leq n, 1\leq k \leq m}|\langle\tau_j , \omega_k\rangle|}{2}\right)	\geq 0, \quad \mbox{for all } x \in \mathcal{Q}_\tau.
	\end{align*}	
\end{corollary}
\begin{corollary}
	Let $\{\tau_j\}_{j=1}^n$,  $\{\omega_j\}_{j=1}^n$ be two orthonormal bases for a  finite dimensional quaternionic Hilbert space $\mathcal{Q}$. Then 
\begin{align*}
	2 \log n \geq S_\tau (x)+S_\omega (x)\geq -2 \log \left(\frac{1+\displaystyle \max_{1\leq j, k \leq n}|\langle\tau_j , \omega_k\rangle|}{2}\right)	\geq 0, \quad \mbox{for all } x \in \mathcal{Q}_\tau.
\end{align*}	
\end{corollary}
Theorem \ref{CDUP} promotes the following question.
\begin{question}
	Let $(\Omega, \mu)$,  $(\Delta, \nu)$ be   measure spaces,    $\mathcal{Q}$ be a quaternionic Hilbert space. For which pairs of 1-bounded continuous Parseval  frames $ \{\tau_\alpha\}_{\alpha\in \Omega}$  and   $\{\omega_\beta\}_{\beta\in \Delta}$ for  $\mathcal{Q}$, we have equality in Inequality (\ref{CDSP})?
\end{question}
In 1988, Maassen  and Uffink (motivated from the conjecture  by Kraus made in 1987 \cite{KRAUS}) improved Deutsch entropic uncertainty principle by using Riesz-Thorin interpolation theorem \cite{MAASSENUFFINK}.
\begin{theorem}\cite{MAASSENUFFINK} 
	(\textbf{Maassen-Uffink Entropic Uncertainty Principle})  \label{MU}
	Let $\{\tau_j\}_{j=1}^n$,  $\{\omega_j\}_{j=1}^n$ be two orthonormal bases for a  finite dimensional Hilbert space $\mathcal{H}$. Then 
	\begin{align*}
		2 \log n \geq S_\tau (h)+S_\omega (h)\geq -2 \log \left(\displaystyle \max_{1\leq j, k \leq n}|\langle\tau_j , \omega_k\rangle|\right), \quad \forall h \in \mathcal{H}_\tau \cap \mathcal{H}_\omega.
	\end{align*}	
\end{theorem}
Based on Theorems \ref{MU} and Theorem \ref{CDUP} we formulate following conjecture.
\begin{conjecture} \textbf{(Quaternionic Kraus Conjecture)}
	Let $(\Omega, \mu)$,  $(\Delta, \nu)$ be   measure spaces  and $\{\tau_\alpha\}_{\alpha\in \Omega}$, $\{\omega_\beta\}_{\beta \in \Delta}$ be  	1-bounded continuous Parseval frames    for a  quaternionic Hilbert  space $\mathcal{Q}$. Then 
		\begin{align*}
			S_\tau(x)+S_\omega (x)\geq -2  \log \left(\displaystyle \sup_{\alpha \in \Omega, \beta \in \Delta}|\langle\tau_\alpha , \omega_\beta\rangle|\right)	\geq 0, \quad \forall x \in \mathcal{Q}_\tau \cap \mathcal{Q}_\omega.	
	\end{align*}
\end{conjecture}

\section{Open Problem}
Breakthrough Crouzeix theorem says following. 
\begin{theorem} \label{CT} \cite{CROUZEIX, CROUZEIX2, CROUZEIX3, BICKELGORKIN, RANSFORDSCHWENNINGER} (\textbf{Crouzeix Theorem}) 
	For every $d \in \mathbb{N}$ and for every matrix $M\in \mathbb{M}_{d}(\mathbb{C})$, we have 
	\begin{align*}
		(\text{\textbf{Crouzeix-Palencia Inequality}})	\quad \quad \quad 	\|p(M)\|\leq (1+\sqrt{2}) \sup\left\{|p(z)|: z \in W_\mathbb{C}(M)\right\},  \quad \forall p \in \mathbb{C}[z].
	\end{align*}	
\end{theorem}

Based on Theorem \ref{CT}, we formulate following problem. 
\begin{problem} (\textbf{Quaternionic Crouzeix Problem})
	 Whether there is a universal constant $R_\mathbb{H}$  satisfying following: For every $d \in \mathbb{N}$ and for every $A\in \mathbb{M}_{d}(\mathbb{H})$, we have 
	\begin{align*}
		\|p(A)\|\leq R_\mathbb{H} \sup\left\{|p(z)|: z \in W_\mathbb{H}(A)\right\},  \quad \forall p(z)\coloneqq a_0+a_1z+a_2z^2+\cdots +a_nz^n \in \mathbb{H}[z], n \in \mathbb{N}
	\end{align*}
	where $\mathbb{H}[z]$ is the set of all polynomials over $\mathbb{H}$. 
We conjecture that $R_\mathbb{H}=8$.
\end{problem}

 \bibliographystyle{plain}
 \bibliography{reference.bib}

\end{document}